\documentclass[10pt, reqno]{amsart}

\usepackage{a4wide}
\usepackage{amsmath, amssymb, amsthm, mathtools, bbm}
\usepackage[outdir=./]{epstopdf}
\usepackage[left=3.2cm, right=3.2cm, bottom=3.2cm]{geometry}
\usepackage[british]{babel}
\usepackage[utf8]{inputenc}
\usepackage[T1]{fontenc}
\usepackage{hyperref}
\usepackage{csquotes}
\usepackage{tikz-cd}
\usepackage{enumitem}
\setenumerate[1]{label={\arabic*)}}

\usepackage{stmaryrd}

\usepackage[backend=biber, style=alphabetic, giveninits=true,maxnames=99,maxalphanames=99]{biblatex}
\usepackage{graphicx}

\newtheorem{theorem}{Theorem}[section]
\newtheorem{lemma}[theorem]{Lemma}
\newtheorem{proposition}[theorem]{Proposition}
\newtheorem{corollary}[theorem]{Corollary}

\theoremstyle{definition}
\newtheorem{definition}[theorem]{Definition}

\theoremstyle{remark}
\newtheorem{remark}[theorem]{Remark}
\newtheorem{example}[theorem]{Example}

\newtheorem{construction}[theorem]{Construction}

\newtheorem{notation}[theorem]{Notation}

\newcommand{\Q}{\mathbb{Q}}
\newcommand{\Z}{\mathbb{Z}}
\newcommand{\N}{\mathbb{N}}
\newcommand{\Cx}{\mathbb{C}}
\newcommand{\CAlg}{\mathrm{CAlg}}
\newcommand{\Sp}{\mathrm{Sp}}
\newcommand{\dr}[3]{\Omega^{#3}_{#1/#2}}

\newcommand{\Alg}{\mathrm{Alg}}

\newcommand{\op}{\mathrm{op}}
\newcommand{\C}{\mathcal{C}}

\DeclareMathOperator{\colim}{\mathrm{colim}}
\newcommand{\colimarg}[1]{\underset{#1}{\colim}}
\newcommand{\Map}{\mathrm{Map}}

\newcommand{\Spec}{\mathrm{Spec}}
\newcommand{\Fil}{\mathrm{Fil}}

\newcommand{\map}{\mathrm{map}}
\renewcommand{\S}{\mathbb{S}}
\newcommand{\Spc}{\mathcal{S}}
\newcommand{\llrrpar}[1]{(\!(#1)\!)}
\newcommand{\essLaur}{\llbracket t^{\pm1}\rrbracket}
\newcommand{\llrrbra}[1]{\llbracket #1\rrbracket}

\newcommand{\gr}{\mathrm{gr}}
\newcommand{\Gr}{\mathrm{Gr}}

\newcommand{\co}{\mathrm{co}}
\renewcommand{\bar}{\mathrm{bar}}

\newcommand{\Ch}{\mathrm{Ch}}

\renewcommand{\O}{\mathcal O}
\newcommand{\unit}{\mathbf{1}}
\newcommand{\aug}{\mathrm{aug}}

\newcommand{\Free}{\mathrm{Free}}

\newcommand{\Poly}{\mathrm{Poly}}

\newcommand{\Mod}{\mathrm{Mod}}
\newcommand{\can}{\mathrm{can}}

\renewcommand{\H}{\mathrm{H}}
\newcommand{\HKR}{\mathrm{HKR}}

\newcommand{\Day}{\mathrm{Day}}

\newcommand{\dR}{\mathrm{dR}}
\newcommand{\an}{\mathrm{an}}
\newcommand{\GL}{\mathrm{GL}}
\newcommand{\R}{\mathbb{R}}
\renewcommand{\unit}{\mathbbm{1}}
\newcommand{\CDGA}{\mathrm{CDGA}}

\newcommand{\F}{\mathcal{F}}
\newcommand{\Shv}{\mathrm{Shv}}

\newcommand{\U}{\mathcal{U}}
\newcommand{\shvHH}{\mathcal{HH}}
\newcommand{\shvHP}{\mathcal{HP}}

\DeclareMathOperator{\HH}{\mathrm{HH}}
\newcommand{\HC}{\mathrm{HC}^-}
\DeclareMathOperator{\HP}{\mathrm{HP}}

\DeclareMathOperator{\Fun}{\mathrm{Fun}}

\newcommand{\lbar}[1]{\mkern 1.5mu\overline{\mkern-1.5mu#1\mkern-1.5mu}\mkern 1.5mu}

\title[Periodic Cyclic Homology over $\Q$]{Periodic Cyclic Homology over $\Q$}
\author[Konrad Bals]{Konrad Bals}
\address{WWU M\"unster, Mathematisches Institut, Einsteinstr. 62, 48149 M\"unster, Germany}
\email{konrad.bals@uni-muenster.de}
\begin{document}

	\begin{abstract}
		Let $X$ be a derived scheme over an animated commutative ring of characteristic 0. We give a complete description of the periodic cyclic homology of $X$ in terms of the Hodge completed derived de Rham complex of $X$. In particular this extends earlier computations of Loday-Quillen to non-smooth algebras. Moreover, we get an explicit condition on the Hodge completed derived de Rham complex, that makes the HKR-filtration on periodic cyclic homology constructed by Antieau and Bhatt-Lurie exhaustive.
	\end{abstract}
	
	\maketitle

	\section{Introduction}
	
	For a commutative ring $k$ and a $k$-algebra $R$, the Hochschild homology $\HH(R/k)$ gives an element in the derived category $D(k)$ of $k$. It has proven itself to be an interesting invariant, appearing for example in trace methods computing algebraic $K$-theory or in Connes' non-commutative geometry. It was also Connes in \cite{Con85} who constructed the cyclic structure on Hochschild homology to define negative cyclic homology $\HC(R/k):=\HH(R/k)^{hS^1}$ and later periodic cyclic homology $\HP(R/k):=\HH(R/k)^{tS^1}$ and proving a relation between $\HC$ of smooth functions on a manifold and de Rham cohomology of the manifold. Transferring Connes' geometric interpretation into algebraic observations in \cite{LQ84} Loday and Quillen compute the homotopy groups $\HC_*(R/k)$ in terms of algebraic de Rham cohomology in many cases. For the purpose of this paper passing here to the Tate-construction, they prove:
	
	\begin{theorem}[\cite{LQ84}]\label{LQ} Assume $\Q\subset k$ commutative and $R$ a smooth commutative $k$-algebra, then
	\[\HP_*(R/k)\cong \bigoplus_{n\in\Z}\H_{\dR}^{*-2n}(R;k)\]
	\end{theorem}

In this paper we give a generalization of this computation to the non-smooth and non-affine situation.\footnote{There recently has been a great development in understanding filtrations on periodic cyclic homology in this generality, or in fact without the rational assumption, e.g. \cite{Ant19}, \cite{Rak20} and \cite{MRT22}, and in particular we will answer exhaustiveness question in the rational case.}
By the classical observation that $\HH(R[S^{-1}]/k)\simeq\HH(R/k)\otimes_RR[S^{-1}]$ for every affine open $\Spec(R[S^{-1}])\subset\Spec R$ Hochschild homology extends to a sheaf $\shvHH_k$ in the Zariski\footnote{In fact by \cite{BMS19} 3.4. even in the fpqc topology via a different argument.} topology on schemes over $k$ (c.f. \cite{GW91}). In fact, similarly we get a sheaf $\shvHP_k$ extending periodic cyclic homology. We recall the details in Appendix \ref{scheme} and write $\HH(X/k):=\Gamma(X,\shvHH_k)$ and $\HP(X/k):=\Gamma(X,\shvHP_k)$.

Moreover, Hochschild homology as a functor $\CAlg_k^\heartsuit\to D(k)$ from discrete $k$-algebras to $D(k)$ is left Kan extended from discrete polynomial algebras\footnote{If $P_\bullet$ is a simplicial resolution of the $k$-algebra $R$, it suffices to check that $|\HH(P_\bullet/k)|\simeq \HH(R/k)$.} and, thus, further extends to a sifted colimit preserving functor from the category of animated commutative (i.e. simplicial commutative) $k$-algebras $\CAlg^\an_{k/}$. So putting both generalizations together and writing $L\dr Xk*$ for the derived de Rham complex of a derived scheme $X$ over an animated $\Q$-algebra $k$, we can state our main theorem in a great generality. In particular, if $k$ is discrete and $X=\Spec(R)$ for a discrete $k$-algebra $R$, this gives new results on the periodic cyclic homology of ordinary algebras.

\begin{theorem}\label{maintheo}
	Given an animated commutative ring $k$ with $\Q\subset\pi_0(k)$ and $X$ a derived $k$-scheme, we have
	\[\HP(X/k)\simeq \prod_{n\in\Z}\widehat{L\dr Xk*}[-2n]\]
	where $\widehat{L\dr Xk*}$ is the completion of $L\dr Xk*$ with respect to the Hodge filtration $L\dr -k{\geq\bullet}$.
\end{theorem}

The key ingredient in the proof is to understand how the Tate-construction behaves under the passage from smooth algebras to general or even animated algebras and it is this behavior that lets the product appear on the right hand side.

In \cite{Ant19} Antieau constructs the HKR-filtration on $\HP(X/k)$ with $n$-th associated graded $\widehat{L\dr Xk*}[2n]$.
If $k$ is an (animated) $\Q$-algebra we can give a complete identification of this HKR-filtration in terms of the equivalence of Theorem \ref{maintheo} and we prove

\begin{theorem}\label{HKR-filtration}
	In the situation of Theorem \ref{maintheo} the HKR-filtration on $\HP(X/k)$ corresponds to the ascending partial product filtration on $\prod_{n\in\Z}\widehat{L\dr Xk*}[-2n]$, that is
	\[\Fil_\HKR^i\HP(X/k)\simeq\prod_{n\leq -i}\widehat{L\dr Xk*}[-2n].\]
	In particular, the HKR-filtration is exhaustive, if and only if $\widehat{L\dr Xk*}$ is (homologically) bounded above.
\end{theorem}

This criterion will give us a large class of examples with exhaustive HKR-filtration. If $k$ is a discrete Noetherian commutative $\Q$-algebra and $X$ an ordinary scheme of finite type over $\Spec\ k$, then Bhatt gives in \cite{Bha12} a concrete way to compute $\widehat{L\dr Xk*}$, which in particular lives in non-positive degrees (cf. Corollary 4.27. \textit{loc.cit.}).

\begin{corollary}\label{Bhatt}
	If $k$ is a discrete Noetherian $\Q$-algebra and $X$ an ordinary finite type scheme over $k$, then the HKR-filtration on $\HP(X/k)$ is exhaustive.
\end{corollary}

Furthermore, the analysis of the Tate-filtration in characteristic 0, which is reviewed in the Appendix \ref{Tatefiltration}, also gives a description of the multiplicativity of the equivalence in the Theorem \ref{maintheo}.  In general for an algebra $A\in\CAlg_k$ there is just no algebra structure on $\prod_{n\in\Z}A[-2n]$, however, for a (animated) commutative $k$-algebra $R$, the object $\HP(R/k)$ carries a natural structure of a commutative algebra in $\Mod_k$. In section \ref{mal} we construct the corresponding multiplicative structure on $\prod_{n\in\Z}\widehat{L\dr Xk*}[-2n]$. On homotopy groups the induced graded ring structure comes from $L\dr Xk{\leq n}\llrrpar t$. Note that there is the terminal topology on $\pi_*\widehat{L\dr Xk{*}}$ making the maps $\pi_*\widehat{L\dr Xk{*}}\to\pi_*L\dr Xk{\leq n}$ continuous. It is not Hausdorff because not every element is detected in some $\pi_* L\dr Rk{\leq n }$. With this we can almost completely describe the graded ring $\pi_*\HP(X/k)$ in terms of $\pi_*\widehat{L\dr Xk{*}}$:
\begin{theorem}\label{algebrastructure}
	In the situation of Theorem \ref{maintheo}, we can describe the homotopy groups $\HP_*(X/k)$ algebraically as
	\[\HP_*(X/k)\cong\left\{\sum_{n\in\Z}a_nt^n:a_n\in\pi_{*+2n}\widehat{L\dr Xk*}\right\}\]
	with addition and multiplication given as
	\begin{align*}
	\left(\sum_{n\in\Z}a_nt^n\right)+\left(\sum_{n\in\Z}b_nt^n\right)&=\sum_{n\in\Z}(a_n+b_n)t^n\\
	\left(\sum_{n\in\Z}a_nt^n\right)\cdot\left(\sum_{n\in\Z}b_nt^n\right)&=\sum_{n\in\Z}c_nt^n
	\end{align*}
	where $c_n$ is a limit of the finite partial sums of $\sum_{i+j=n}a_i\cdot b_j$ in the topology on $\pi_{*}\widehat{L\dr Rk{*}}$\footnote{This is sometimes called a net and explicitly means for every open $U\ni 0$, there is a finite subset $I_0\subset\{i+j=n\}$, such that for all finite subset $J\subset\{i+j=n\}$ containing $I_0$
	we have $c_n-\sum_{(i,j)\in J}a_i\cdot b_j\in U.$}
\end{theorem}
However, we want to immediately issue the warning that because the topology on $\pi_*\widehat{L\dr Rk*}$ is not Hausdorff, the element $c_n\in\pi_*\widehat{L\dr Rk*}$ is not uniquely determined as a limit. To fully understand the homotopy groups $\HP_*(X/k)$ algebraically, one, furthermore, has to analyze the $\lim^1$-terms contributing to $\pi_*\widehat{L\dr Rk*}$.

\subsection{Outline}

We begin in section \ref{formal section} with a formality statement for $S^1$-actions in the derived category over rational algebras (Corollary \ref{formalaction}) in order to recall a coherent version of the HKR-theorem in Proposition \ref{HKR}. This allows us to coherently compute $\HP$ for smooth algebras.

In section \ref{main section} we will use the language of filtrations in order to generalize the computations for smooth algebras to arbitrary derived schemes and prove Theorem \ref{maintheo} (cf. Theorem \ref{main theorem}). In particular we will use the multiplicativity of the Tate-filtration. The Tate-filtration itself and its multiplicative structure in the rational setting will be reviewed in the Appendix \ref{Tatefiltration}. Furthermore in section \ref{main section} we will exploit the consequences for the HKR-filtration and prove Theorem \ref{HKR-filtration} and Corollary \ref{Bhatt}. Finally, the last section (\ref{mal}) is completely devoted to the proof of Theorem \ref{algebrastructure}.

\subsection{Notation}

Throughout this note we are freely using the $\infty$-categorical language as developed in \cite{Lur09} and \cite{Lur16}. In particular, for a commutative ring $k$ we identify the derived category $D(k)$ with the category $\Mod_k:=\Mod_{\H k}{\Sp}$ of $\H k$-module spectra and thus view it as a stably symmetric monoidal $\infty$-category. It comes with a canonical lax symmetric monoidal functor $\iota\colon \Ch_*(k)\to D(k)$ from the 1-category of chain complexes and we will constantly abuse notation by identifying $C_*$ with $\iota C_*$ for $C_*\in\Ch_*(k)$.

Moreover, we will use the 1-category $\CDGA_k$ of commutative differential graded algebras over $k$. An object $(C_*,d)\in \CDGA_k$ consists of a commutative graded $k$-algebra $\bigoplus_{i\in\Z} C_i$ of discrete $R$-modules with differentials $d\colon C_{i-1}\to C_{i}$ for all $i>0$ satisfying the Leibniz rule. There will be two orthogonal ways to view a $\CDGA_k$ as an object in $\CAlg_k$, either with 0 differential or with differential $d$ and we already warn the reader to not confuse those functors.

In particular, for a commutative ring $k$ and a commutative $k$-algebra $R$, we will generally view the de Rham complex $\dr Rk*$ as an object in $\CAlg_k:=\CAlg(\Mod_k)$, and if we want to view it as a CDGA over $k$ we write $\dr RkH$.

Later in the paper, we need to talk about filtrations in a stable category $\C$, by which we always mean decreasingly indexed, $\Z$-graded filtrations, i.e. functors from $\Z_\leq^\op$ into $\C$. For a symmetric monoidal category $\C$ we equip the category $\Fil(\C):=\Fun(\Z_\leq^\op,\C)$ with the Day convolution tensor product $\otimes^\Day$. The $n$-th associated graded $\gr^n F$ of $F$ is given by the cofibre of the map $F^{n+1}\to F^{n}$. A splitting of a filtration $F^\bullet\in\Fil(\C)$ consists of a collection $(A_n)_{n\in\Z}$ together with an map of filtrations $\bigoplus_{n\geq\bullet}A_n\to F^\bullet$ inducing an equivalence on associated graded. In particular, a splitting $(A_n)$ of $F$ canonical gives an identification $\gr^n F\simeq A_n$.
Finally, to fix vocabulary, a filtration $F^\bullet\in\Fil(\C)$ on $F\in\C$ is complete if $\lim F^\bullet\simeq 0$ and is exhaustive if $\colim F^\bullet\simeq F$. We write $\Fil^\wedge(\C)\subset \Fil(\C)$ for the full subcategory on complete filtrations and denote by $(-)^\wedge$ its left adjoint.

\subsection{Acknowledgment}
I would like to thank Achim Krause, Jonas McCandless and Thomas Nikolaus for helpful discussions on this topic. I also want to thank Zhouhang \textsc{Mao} for making remarks on an earlier version. Finally, again I want to thank Thomas Nikolaus for bringing this project up. Funded by the Deutsche Forschungsgemeinschaft (DFG, German Research Foundation) under Germany’s Excellence Strategy EXC
2044–390685587, Mathematics Münster: Dynamics–Geometry–Structure and the CRC 1442 Geometry: Deformations and Rigidity.

\section{Formality over \texorpdfstring{$\Q$}{Q}}\label{formal section}

The explicit computations heavily rely on strong formality properties that hold if working over $\Q$-algebras. In this section we will prove a strong version of the HKR-theorem for Hochschild homology. This enables us to establish a coherent versions of the Theorem \ref{LQ} copied from \cite{LQ84}. These results are not new and can be found in \cite{TV11} and \cite{MRT22}, but we would like to present them here.

Throughout the first section, let $k$ be a discrete commutative $\Q$-algebra.
The key ingredient is a formality statement of $C_*(S^1,k)$, due to \cite{TV11}.

\begin{construction}\label{A}
	The multiplication $S^1\times S^1\to S^1$ and the diagonal $S^1\to S^1\times S^1$ exhibit $S^1$ as an associative bialgebra in spaces. Because the symmetric monoidal structure on $\Spc$ is cartesian, by the dual of \cite{Lur15}[Proposition 2.4.3.8] the coalgebra structure given by the diagonal refines to a cocommutative coalgebra structure. Now the functor $C_*(-,k)\colon \Spc\to D(k)$ from spaces to the derived category of $k$ taking singular chains with coefficients in $k$ refines via the Eilenberg-Zilber maps to a symmetric monoidal functor. Therefore $C_*(S^1,k)$ acquires the structure of a cocommutative bialgebra in $D(k)$.
	
	Moreover, the functor $\iota\colon\Ch_*(k)\to D(k)$ from the 1-category of chain complexes to the $\infty$-category $D(k)$ is lax symmetric monoidal and precisely restricts to a symmetric monoidal functor on the full 1-subcategory $\Ch_*^{K-\mathrm{flat}}(k)$ of K-flat chain complexes. Thus the chain complex for $\epsilon$ in degree 1
	\[\Lambda_k(\epsilon):=(k\cdot \epsilon\xrightarrow 0 k\cdot 1)\]
	with multiplication $\epsilon^2=0$ and primitive comultiplication $\Delta \epsilon=(\epsilon\otimes 1+1\otimes\epsilon)$ gives a cocommutative bialgebra object in $D(k)$ under the identification of $\Lambda_k(\epsilon)$ as an element in $D(k)$.
\end{construction}

\begin{proposition}[\cite{TV11}]\label{formal}
	In this setting where $k$ is a discrete $\Q$-algebra, there is a natural equivalence $C_*(S^1,k)\simeq \Lambda_k(\epsilon)$  as cocommutative bialgebras in $D(k)$ for $\epsilon$ primitive in degree 1.
\end{proposition}
For completeness reasons we would like to include a proof here:

\begin{proof}
	Both objects $C_*(S^1,k)$ and $\Lambda_k(\epsilon)$ have canonical augmentations coming from $S^1\to *$ in $\Spc$ and $\epsilon\mapsto 0$ in $\Ch_*(k)$. We will in fact show, that they even agree as augmented cocommutative algebras in $D(k)$.
	Using the adjunction (e.g. cf. \cite{Lur16} Theorem 5.2.2.17\footnote{There is a gap in the proof of the cited reference as pointed out by \cite{DH22}, which could be fixed in the latest version (v4) of \cite{BCN23}.})
	\[\begin{tikzcd}\bar\colon\Alg^\aug(\co\CAlg D(k))\ar[r, shift left = 1]&\co\CAlg^\aug (D(k)):\!\co\bar\ar[l, shift left=1]\end{tikzcd}\] 
	 it satisfies to construct a map of (co-)augmented cocommutative coalgebras under the bar-functor. In fact the computation in \cite{Ada56} show that for $C_*(S^1,k)$ the unit of the adjunction $C_*(S^1,k)\to\co\bar(\bar C_*(S^1,k))$ is an equivalence, so that an identification of $\bar C_*(S^1,k)\simeq C_*(BS^1,k)$ translates to an identification of $C_*(S^1,k)$ under $\co\bar$. Therefore, we want to understand the cocommutative coalgebra structure of $\bar C_*(BS^1,k)$ or equivalently the dual commutative algebra structure on $C^*(BS^1,k)$, as both objects are of finite type. 
	A choice of a generator in $H^2(BS^1,k)$ gives a map $k[x]:=\Free(k[-2])\to C^*(BS^1,k)$ from the free commutative $k$-algebra on a generator $x$ in degree $-2$. Because $\Q\subset k$, on homotopy groups both sides are free on a generator in degree $-2$ and we have $C^*(BS^1,k)\simeq k[x]$ is free as a commutative algebra. Finally, translating back to cocommutative bialgebras, we can compute
	\[\co\bar (k[x])^\vee\simeq(\bar k[x])^\vee\simeq \left(k\otimes_{k[x]} k\right)^\vee\]
	by resolving $k$ with the DGA $(\Lambda_{k[x]}(\epsilon^\vee), d\epsilon^\vee=x)$ for a primitive element $\epsilon^\vee$ in degree $-1$. Thus $\left(k\otimes_{k[x]} k\right)^\vee\simeq\Lambda_k(\epsilon^\vee)^\vee\simeq\Lambda_k(\epsilon)$ for $\epsilon$ a dual basis to $\epsilon^\vee$.
\end{proof}

From now on, to shorten notation we set $A:=\Lambda_k(\epsilon)$ for $|\epsilon|=1$ primitive.

\begin{corollary}\label{formalaction}
	For a rational discrete algebra $k$ the categories $\Fun(BS^1,D(k))$ and $\Mod_{A}D(k)$ are equivalent as symmetric monoidal categories, where the symmetric monoidal structure on the latter comes from the coalgebra structure on $A$.
\end{corollary}
\begin{proof}
	There is a symmetric monoidal equivalence $\Fun(BS^1,D(k))\simeq \Mod_{C_*(S^1,k)}D(k)$ as symmetric monoidal categories, where the symmetric monoidal structure on the right hand side comes from the cocommutative bialgebra structure on $C_*(S^1,k)$. Thus the equivalence $C_*(S^1,k)\simeq A$ as cocommutative bialgebras gives a symmetric monoidal equivalence of their module categories (c.f. Proposition 2.2.1. in \cite{Rak20}).
\end{proof}

\begin{remark}
	The above equivalence induces the identity on underlying objects. Thus, given a complex $X\in D(k)$ equipping $X$ with an action of $S^1$ is equivalent to providing a module structure over $A$. Informally, this amounts to a map $d\colon k\cdot \epsilon[1]\otimes X\simeq X[1]\to X$ and coherent homotopies witnessing $d^2\simeq 0$.  
\end{remark}

\begin{construction}
	Let $\CDGA_k$ denote the 1-category of commutative differential graded algebras over $k$ as introduced in the Notations. Forgetting the differential, there is a functor $\CDGA_k\to\CAlg\Ch_*(k)$ sending $(C_*,d)\in \CDGA$ to $\bigoplus_{i\in\Z} C^i[i]\in\CAlg\Ch_*(k)$ with 0 differential. In 1-categories now an action of $A$ precisely corresponds to an ascending differential, such that this functor refines through $\CAlg\Mod_A\Ch_*(k)$ and postcomposing with $\iota$ we get a map
	\[\CDGA_k\to\CAlg\Mod_A\Ch_*(k)\to\CAlg\Mod_A D(k).\]
	To avoid confusion we will write $U\colon\CDGA_k\to\CAlg_k^{BS^1}$ for this functor.\label{deRham}
\end{construction}
	
\begin{remark}	
	For a $k$-algebra $R$ the de Rham complex $\dr RkH$ by definition lives in $\CDGA_k$. Now via the previous construction the underlying chain complex 
	\begin{equation}\label{derham}U\dr RkH\simeq\prod_{n\in\N}\dr Rkn[n]\simeq\left(\cdots\xrightarrow 0\dr Rk2\xrightarrow 0 \dr Rk1\xrightarrow 0 \dr Rk0\right)\end{equation}
	gives an object in $\CAlg_k^{BS^1}$.
\end{remark}

This simplifies the analysis of Hochschild homology in the rational setting and we can phrase a strong version of the HKR-theorem, which has been well known (e.g. \cite{Qui70}). However, we would like to emphasize on all the structure the following result captures and give a different proof as in the cited source:

\begin{proposition}\label{HKR}
	If $\Q\subset k$, then for every smooth discrete $k$-algebra $R$, there are natural equivalences
	\[\HH(R/k)\xrightarrow{\sim}U\dr RkH\simeq\prod_{n\in\N}\dr Rkn[n]\]
	of commutative algebras in $D(k)$ with $S^1$-action, where the $S^1$-action on the right hand side is given by the de Rham differential (cf. Construction \ref{deRham}).
\end{proposition}
\begin{proof}
	In the category $\CAlg_k^{BS^1}$ Hochschild homology enjoys a universal property: For every commutative $k$-algebra $S$ with $S^1$-action any non-equivariant map $R\to S$ extends uniquely up to contractible choice over $R\to\HH(R/k)$. Thus we get the dashed $S^1$-equivariant algebra map
	\[\begin{tikzcd}
	R\simeq\dr Rk0\ar[r]\ar[d]&\prod_{n\in\N}\dr Rkn[n]\\
	\HH(R/k)\ar[ru, dashed]
	\end{tikzcd}\]
	 The original computation in the HKR-theorem \cite{HKR62} gives an equivalence $\dr RkH\xrightarrow{\sim}\HH_*(R/k)$ of differentially graded algebras. Postcomposing with the map above on homotopy groups, we get a map $\dr RkH\to\HH_*(R/k)\to\dr RkH$. Finally, $\dr RkH$ has a universal property among commutative differentially graded algebras, as the initial CDGA with a map from $R$ into its zeroth part. Because on the zeroth part the composition above is given by the identity $R\to R$, the same is true for the entire map, forcing $\HH_*(R)\to\dr RkH$ to be an equivalence.
\end{proof}

\begin{remark}
	This strong version of the HKR-theorem can be understood as a rigidification of the Hochschild homology functor from polynomial $k$-algebras $\Poly_k$: It gives a functorial factorization 
	\[\begin{tikzcd}
	& \CDGA_k\ar[d,"U"]\\
	\Poly_k\ar[ru, "\dr -kH"]\ar[r,swap, "\HH(-/k)"]&\CAlg\Mod_{A} D(k),
	\end{tikzcd}
	\]
	through the functor $\dr -kH\colon\Poly_k\to\CDGA_k$ of 1-categories.
\end{remark}

We can now get a very good understanding of the Tate-construction for such formal objects:

\begin{definition}
	For $(C_*,d)\in\CDGA_k$ we write $|C_*|$ for the chain algebra $(C_{-*},d)$. This gives a functor $|-|\colon\CDGA_k\to\CAlg\Ch_*(k)$ of 1-categories. More generally, given a graded object $C_*$ with differential $d$ we want to write $|C_*|$ to stress that we view it as a chain complex.
\end{definition}

\begin{example}
	By definition we have $|\dr RkH|\simeq\dr Rk*$.
\end{example}

With this notation set, we can make the classical computations of periodic cyclic homology in characteristic zero. This is also done for example in the lectures \cite{KN18}.

\begin{lemma}\label{computation}
	For $(C_*,d)\in\CDGA_k$ there is a natural map $|C_*|\to (UC_*)^{tS^1}$ in $\CAlg_k$.
\end{lemma}
\begin{proof}
	Because of the lax monoidal natural transformation $(-)^{hS^1}\to (-)^{tS^1}$, it suffices to establish a natural map $|C_*|\to C_*^{hS^1}$. Under the symmetric monoidal equivalence $\Fun(BS^1,D(k))\simeq\Mod_{A} D(k)$ the functor $(-)^{hS^1}$ corresponds to $\map_{A}(k,-)$. A choice of projective resolution $P_*$ of $k$ as an $A$-coalgebra reduces us to give a functorial map $|C_*|\to\map_{A}(P_*,UC_*)$ where the right hand side is the 1-categorical mapping chain complex. Now put $P_*=(A\langle t^\vee\rangle, d_P)$ as the free divided power algebra on a primitive generator $t^\vee$ in degree $2$ with $d_Pt^\vee=\epsilon$. Thus, computing the mapping chain complex gives an equivalence
	\[\map_A(P_*,UC_*)\cong (UC_*\llrrbra{t},t d)\] 
	for $|t|=-2$ a dual generator to $t^\vee$ and we can explicitly describe a multiplicative chain map $|C_i|\to (UC_*\llrrbra t,td)$ given by $C_i\simeq C_i\cdot t^i\to UC_*\llrrbra t$ on chain groups. This finishes the proof.
\end{proof}
\begin{remark}
	The computation in Lemma \ref{computation} actually completely describes $UC_*^{hS^1}$ and under the equivalence $UC_*^{tS^1}\simeq UC_*^{hS^1}\otimes_{k^{hS^1}}k^{tS^1}$ we already get a full identification $UC_*^{tS^1}\simeq (UC_*(\!(t)\!),td)$.
	
	For $C_*=\dr RkH$ for $R$ smooth over a rational algebra $k$ we thus could have a full understanding of $\HP(R/k)$. However, we will not directly use this, but rather proof a general statement with more structure, that generalizes to non-smooth and animated algebras.
\end{remark}

\section{Main Theorem}\label{main section}

\begin{notation}
	Given $C_*\in\CDGA_k$, we denote by $\Fil^\bullet_H|C_*|$ the filtration
	\[\Fil^n_H|C_*|:=|\tau_{\geq n}C_*|\]
	where $\tau_{\geq n}C_*$ is the part of grading greater or equal $n$.
	Unraveling, $\Fil^\bullet_H|C_*|$ precisely gives the stupid or brutal filtration on the chain complex $|C_*|\in\Ch_*(k)$.
	
	Moreover, for $X\in\Fun(BS^1,\Sp)$ let $\Fil^\bullet_T X^{tS^1}$ be the Tate filtration on $X^{tS^1}$, see Appendix \ref{Tatefiltration} for more details. It is a complete commutatively multiplicative and exhaustive filtration with associated graded $\gr^n\Fil_T X^{tS^1}\simeq X[-2n]$. The Tate-filtration also restricts to a complete (and exhaustive) filtration on $\Fil_T^0X^{tS^1}\simeq X^{hS^1}$.
\end{notation}

\begin{theorem}\label{generalfiltration}
	For $C_*\in\CDGA_k$ the map $|C_*|\to UC_*^{tS^1}$ refines and extends to an equivalence
	\[\left(\Fil^\bullet_H|C_*|\otimes^\Day\Fil^\bullet_Tk^{tS^1}\right)^\wedge\to\Fil^\bullet_TUC_*^{tS^1}\]
	of commutatively multiplicative filtered objects in $D(k)$.
\end{theorem}
\begin{proof}
	In the concrete description of $C_*^{hS^1}$ in Lemma \ref{computation}, we can identify the Tate-filtration with the $t$-adic filtration on $(C^*\llrrbra t,td)$ via Proposition \ref{t-adic filtration} and the map from Lemma \ref{computation} refines to a map of commutatively multiplicative filtrations $\Fil^\bullet_H|C_*|\to\Fil_T^\bullet UC_*^{tS^1}$. Because the target is a module over the commutative algebra $\Fil^\bullet_Tk^{tS^1}$, we get the map 
	\begin{equation}\Fil^\bullet_H|C_*|\otimes^\Day\Fil^\bullet_Tk^{tS^1}\to\Fil^\bullet_TUC_*^{tS^1}\label{map}\end{equation}
	and because the target is complete, it even factors over the completion. To show that we get an equivalence of complete filtrations, it is enough to check on associated graded. Let us introduce a formal character $t$ in degree $-2$ to visually relate Tate filtrations and $t$-adic filtrations and write $\gr^n\Fil^\bullet_Tk^{tS^1}\simeq k[-2n]=:k\cdot t^n$. Then on the $n$th associated graded the map \eqref{map} is given by 
	\[\bigoplus_{i+j=n} C_i[-i]\otimes k\cdot t^j\simeq\bigoplus_{i+j=n} C_i[i]\cdot t^i\otimes k\cdot t^j\to UC_*\cdot t^n\]
	and thus an equivalence by construction, as $UC_*\simeq \bigoplus C_i[i]$ as a complex.
\end{proof}

We can finally return to our situation of interest and immediately get a description of $\HP(R/k)$ in more general situations:

\begin{corollary}\label{main filtration}
	If $k$ is an animated ring with rational homotopy groups and $R$ in $(\CAlg^\an)_{k/}$, then there is an equivalence of commutatively multiplicative complete filtrations
		\begin{equation}\label{filtration}\left(\Fil_H^\bullet L\dr Rk*\otimes^\Day\Fil_T^\bullet k^{tS^1}\right)^\wedge\to \Fil_T^\bullet\HP(R/k).\end{equation}
\end{corollary}
\begin{proof}
	First assume that $k$ is discrete. We want to show that both sides commute with sifted colimits as functors to $\Fil^\wedge(D(k))$. For $\Fil_H^\bullet L\dr Rk*$ after completion this is by definition and because the Day convolution tensor product commutes with all colimits it follows for the left hand side. As functors to complete filtrations we can also check this on associated graded for the right hand side: And also here any shifts of $\HH(R/k)$ commute with sifted colimits.
	
	We thus can reduce to the case that $k$ is an ordinary $\Q$-algebra and $R$ smooth over $k$. Then the equivalence immediately follows from Theorem \ref{generalfiltration} by putting $C_*=\dr RkH$.

	In the general case of an animated morphism $k\to R$ between animated $\Q$-algebras we can give the exact same proof. Choose a simplicial resolution $k_n\to R_n$ of polynomial algebras. Again by definition $\Fil_H^\bullet L\dr Rk*\simeq\colim \Fil_H^\bullet L\dr {R_n}{k_n}*$ and thus the left hand side is determined by its value on polynomial rings. On the right hand side we check again, that on associated graded we get an equivalence
	\[\HH(R/k)\simeq\HH(R/\Q)\otimes_{\HH(k/\Q)}k\simeq\colim \HH(R_n/\Q)\otimes_{\HH(k_n/\Q)} k_n\] 
	where the first equivalence comes from the base-change formula for Hochschild homology (cf. \cite{AMN17} proof of Theorem 3.4) and the second from the facts that $\HH(-/\Q)$ commutes with colimits in $\CAlg_\Q$ and that the colimit is sifted. Thus also in the general case, the statement reduces to Theorem \ref{generalfiltration}.
\end{proof}

Finally, in order to compute the periodic cyclic homology in our case, we only have to understand the left hand filtration in \eqref{filtration}. There are basically two obstacles, that we have to take care of: Completion does not behave well with Day convolution and does not behave well with underlying objects.
\begin{theorem}\label{main theorem}
	Let $k$ be an animated ring with $\Q\subset \pi_0 k$ and $X$ a derived scheme over $k$. Then there is a natural equivalence of underlying objects in $\Mod_k$
	\[\HP(X/k)\simeq \prod_{n\in\Z}\widehat{L\dr Xk*}[-2n]\] 
\end{theorem}
\begin{proof}
	Because both sides are sheaves in the Zariski topology on $X$ we are reduced to the case $X=\Spec R$ for $R\in\CAlg^{\an}_{k/}$.
	By the above Corollary \ref{main filtration} there is a natural equivalence of filtrations
	\[\left(\Fil_H^\bullet L\dr Rk*\otimes^\Day\Fil_T^\bullet k^{tS^1}\right)^\wedge\to \Fil_T^\bullet\HP(R/k).\]
	Because the Tate-filtration is exhaustive on $\HP(R/k)$ it suffices to compute the underlying object of the left filtration. 
	Now the filtration $\Fil_Tk^{tS^1}$ carries a canonical splitting, because the connecting homomorphism in
	\[\begin{tikzcd}
	\Fil^{n+1}_T k^{tS^1}\ar[r]\ar[d, equal]&\Fil^n_T k^{tS^1}\ar[r]\ar[d, equal]&\gr^n\Fil_T^\bullet k^{tS^1}\ar[d, equal]\\
	k^{hS^1}[-2(n+1)]\ar[r]&k^{hS^1}[-2n]\ar[r]&k[-2n]
	\end{tikzcd}\]
	is forced to vanish for degree reasons, in fact $\Map(k[-2],k^{hS^1}[-2n-3])$ is contractible. Therefore, we have a map of filtrations $\bigoplus_{n\geq\bullet}k[-2n]\to\Fil^\bullet_Tk^{tS^1}$, inducing an equivalence on associated graded, and, thus, as the left hand side is complete, it even is an equivalence of filtrations.
	
	We claim now, that this splitting induces an equivalence \[\prod_{n\in\Z}(\Fil_H^{\bullet-n} L\dr Rk*[-2n])^\wedge\simeq(\Fil_HL\dr Rk*\otimes^\Day\Fil_Tk^{tS^1})^\wedge\]
	Indeed, the canonical map 
	$\bigoplus_{n\in\Z}(\Fil_H^{\bullet-n} L\dr Rk*[-2n])\to\prod_{n\in\Z}(\Fil_H^{\bullet-n} L\dr Rk*[-2n])^\wedge$
	exhibits the right hand side as the completion: It is evidently complete and the map on the $m$-th associated graded
	\[\bigoplus_{n\in\Z}L\dr Rk{m-n}[-2n]\to\prod_{i\in\Z} L\dr Rk{m-n}[-2n]\]
	is an equivalence, because $L\dr Rk{m-n}$ is always bounded below and 0 for $n>m$.
	
	Finally we want to compute the underlying object, i.e. the colimit. Consider the canonical colimit-limit-interchange $\can$ map sitting in the cofibre sequence
	\[\colim\left(\prod_{n\in\Z}\Fil_H^{\bullet-n}\widehat{L\dr Rk*}[-2n]\right)\xrightarrow\can\prod_{n\in\Z}\widehat{L\dr Rk*}[-2n]\to\colim\left(\prod_{n\in\Z}L\dr Rk{\leq\bullet -n -1}[-2n]\right)\]
	But because $L\dr Rk{\leq\bullet- n-1}$ is bounded below for all $n$, and 0 for $n\geq\bullet$, the right most product is actually degreewise finite and, thus, vanishes in the colimit. Now putting everything together gives the result.
\end{proof}

We want to use the result to investigate the exhaustiveness of the $\HKR$-filtration constructed in \cite{Ant19}. It arises from the left Kan extension of the Beilinson Whitehead tower of the Tate filtration on $\HP(-/k)$ from smooth algebras to bicomplete filtrations as the underlying outer filtration. For more details c.f. \textit{loc. cit.} or \cite{BL22} section 6.3. 

\begin{proposition}\label{comparison}
	In the situation of the Theorem \ref{main theorem}, the $\HKR$-filtration on $\HP(R/k)$ can be identified with the filtration by partial products of $\prod_{n\in \Z}\widehat{L\dr {R_n}{k_n}*}[-2n]$. Precisely, \[\Fil^i_{\HKR}\HP(R/k)\simeq\prod_{n\leq -i}\widehat{L\dr {R_n}{k_n}*}[-2n]\]
\end{proposition}
\begin{proof}
	By definition of the $\HKR$-filtration we only have to construct equivalences in the case $R$ over $k$ a smooth algebra. 
	But now the Tate-filtration on $\HP(R/k)$ induces a shifted Hodge filtration on the factor $\dr Rk*[2n]$ with $\Fil^m_T(\dr Rk*[2n])\simeq (\Fil_H^{n+m}\dr Rk*)[2n]$. Because $\Fil_H^{n+m}\dr Rk*\in D(k)_{\leq -n-m}$ we have
	\[\Fil^{m}_T(\dr Rk*[2n])\in D(k)_{\leq n-m}\]
	Moreover, we can similarly compute
	\[\gr^m\Fil_T^\bullet(\dr Rk*[2n])\simeq \dr Rk{n+m}[-n-m+2n]\in D(k)_{\geq n-m}
	.
	\]
	In fact these two conditions precisely show that $\Fil_T^\bullet(\dr Rk*[2n])$ is concentrated in degree $n$ with respect to the Beilinson $t$-structure on $\Fil(D(k))$. From our complete description of $\Fil_T^\bullet\HP(R/k)$ in terms of $\dr Rk*\cdot [2n]$ we get
	\[\Fil_T^\bullet(\dr Rk*[2n])\simeq\pi_n\Fil_T^\bullet\HP(R/k)\simeq\gr^n\Fil_\HKR^\bullet\HP(R/k)\]
	where the last equivalence comes form the definition of the $\HKR$-filtration. In particular, $\HP(R/k)$ decomposes into the product of the associated gradeds of the $\HKR$-filtration, which proves the claim.
\end{proof}

\begin{corollary}
	In the situation of the theorem the $\HKR$-filtration from \cite{Ant19} is exhaustive if and only if $\widehat{L\dr Xk*}$ is bounded above.
\end{corollary}
\begin{proof}
	We can phrase the exhaustiveness as the condition that the natural map
	\[\colimarg i\prod_{n\geq i}\widehat{L\dr Xk*}[2n]\to \prod_{n\in\Z}\widehat{L\dr Xk*}[2n]\simeq\prod_{n\in\Z}\widehat{L\dr Xk*}[-2n]\]
	is an equivalence. This is precisely the case when $\widehat{L\dr Xk*}[2n]$ eventually leaves any fixed degree for $n\to\infty$, precisely when it is bounded above.
\end{proof}

\begin{example}
	In \cite{Ant19} Antieau proves without assumptions on the discrete commutative base ring $k$, that the $\HKR$-filtration is exhaustive if $X$ is quasi-lci over $k$, i.e. $L\dr Rk1$ has Tor-amplitude in $[0,1]$. We recover this statement in our situation via the observation that the lci-condition forces $\widehat{L\dr Xk*}$ to be concentrated in degrees $(-\infty,0]$. 
\end{example}
	
	Moreover, with the result in \cite{Bha12} in the rational setting we can even prove a more drastic result:

\begin{corollary}
	If $k$ is a discrete Noetherian $\Q$-algebra and $X$ a locally finite type scheme over $k$, then the $\HKR$-filtration is exhaustive.
\end{corollary}
\begin{proof}
	The value of $\Gamma(X,\widehat{L\dr -k*})$ can be computed as a limit from the value on affines. But by the explicit description \cite{Bha12} Proposition 4.10. the Hodge completed derived de Rham complex is concentrated in non-positive degree on affines, thus the same is true for the limit.
\end{proof}

\section{Multiplicative Structure}\label{mal}
	In the Corollary \ref{main filtration} the equivalence 
	\[\left(\Fil_H^\bullet L\dr Rk*\otimes^\Day\Fil_T^\bullet k^{tS^1}\right)^\wedge\to \Fil_T^\bullet\HP(R/k)\]
	was compatible with the commutative algebra structures on both sides. Thus we are able to deduce properties of the induced commutative algebra structure on $\prod_{n\in\Z}\widehat{L\dr Rk*}[-2n]$. But first we will describe algebra structures on these big products more generally:

\begin{definition}
	Given a complete and exhaustive commutative multiplicative filtration $R^\bullet\in\CAlg\Fil(\Mod_k)$ on a commutative algebra $R\in\CAlg_k$. We define 
	\[R\essLaur :=\colim\left(R^\bullet\otimes^\Day\Fil_T^\bullet k^{tS^1}\right)^\wedge\]
\end{definition}

\begin{example}\label{example}
	If $R\in\CAlg_k$ for an animated commutative ring $k$, equipped with the constant negatively graded filtration, then we have $R\essLaur\simeq R^{tS^1}$ with respect to the trivial $S^1$-action on $R$. If moreover, $\pi_0k$ is rational, we can even write $R\llrrpar t:= R^{tS^1}$ as the unique commutative algebra in $\Mod_k$ with homotopy groups $\pi_*R\llrrpar t$ for a generator $|t|=-2$.
\end{example}

\begin{corollary}\label{maincor}
	In the situation of Theorem \ref{main theorem} the equivalence refines to a natural equivalence $\HP(X/k)\simeq \widehat{L\dr Xk*}\essLaur$ in $\CAlg_k$.
\end{corollary}

In fact, in the situation of Corollary \ref{maincor} we can demystify the object $\widehat{L\dr Xk*}\essLaur$. 
The object $R\essLaur$ does not fully depend on $R^\bullet$ as a complete filtered object. We will show a very special case, of this feature:

\begin{lemma}\label{middle}
	If $F^\bullet\in \Mod_k$ is a filtered object with $F^n=0$ for $n$ but finite $n$, then $\colim(F^\bullet\otimes^\Day\Fil_T^\bullet k^{tS^1})^\wedge \simeq 0$. In particular, if $R^\bullet\to \lbar R^\bullet$ is a map in $\CAlg\Fil(\Mod_k)^\wedge$ such that the maps induce equivalences for all but finite $n$, then $R\essLaur\simeq \lbar R\essLaur$.
\end{lemma}

\begin{proof}
	As in the proof of Theorem \ref{main theorem}, we get an equivalence
	$\bigoplus_{n\in\Z}F^{\bullet-n}[-2n]\xrightarrow\sim F^\bullet\otimes^\Day\Fil_T^\bullet k^{tS^1}$. However, the left hand side is already complete: $F^{\bullet-n}$ is complete because it is eventually 0 and the direct sum is in fact a product, because there are only finitely many non-zeros factors. Finally, the underlying object of $F^\bullet $ is 0 and thus also of the complete filtration $F^\bullet\otimes^\Day\Fil_T^\bullet k^{tS^1}$.
	
	For the last statement, we note that the construction $\colim(-\otimes^\Day\Fil_T^\bullet k^{tS^1})^\wedge$ is exact.
\end{proof}

\begin{proposition}\label{com}
	In Corollary \ref{maincor} we can have further identifications of commutative algebras $\HP(X/k)\simeq\widehat{L\dr Xk*}\essLaur\xrightarrow\sim\lim_m L\dr Xk{\leq m}\llrrpar t$.
\end{proposition}
\begin{proof} 
	We start in a general setting: Given a complete multiplicative exhaustive filtration $R^\bullet$ on a $k$-algebra $R$. Set $R^\bullet/R^m$ to be the filtration with $(R^\bullet/R^m)(l):= R^l/R^m$ for $l\leq m$ and 0 otherwise. Because $R^\bullet$ is complete we have $R^\bullet\xrightarrow\sim\lim_m R^\bullet/R^m$. Checking on associated graded we get an equivalence
	\[\left(R^\bullet\otimes^\Day\Fil^\bullet_T k^{tS^1}\right)^\wedge\xrightarrow\sim\lim_m\left(R^\bullet/R^m\otimes^\Day\Fil^\bullet_T k^{tS^1}\right)^\wedge\]
	and thus the natural map $R\essLaur \to \lim_m (R/R^m\essLaur)$. Now if $R^\bullet$ is eventually constant in negative degrees, and because it is eventually 0 in positive degrees $R^\bullet/R^m\essLaur\simeq R/R^m\llrrpar t$ by Lemma \ref{middle} and Example \ref{example}.
	
	Finally, in the concrete situation $R^\bullet= \Fil_H^\bullet\widehat{L\dr Xk*}$, which satisfies this last assumption, we have an easy description of the quotients $\widehat{L\dr Xk{*}}/\widehat{L\dr Xk{\geq m+1}}\simeq L\dr Xk{\leq m}$. And now the proof of Theorem \ref{main theorem} gives an equivalence $L\dr Xk{\leq m}\essLaur\simeq \prod_{n\in \Z} L\dr Xk{\leq m}[-2n]$ on underlying objects, such that the map from $\widehat{L\dr Xk{*}}\essLaur$ can be identified with the natural map $\widehat{L\dr Xk{*}}\to L\dr Xk{\leq m}$ in each factor. In particular this map is an equivalence in the limit.
\end{proof}

We can finally get to the description of the homotopy groups $\HP_*(X/k)$ explained in the introduction. Disregarding the multiplicative structure on $\HP_*(X/k)$ Theorem \ref{main theorem} already gives the additive identification
\[\HP_*(X/k)\cong\prod_{n\in\Z}\pi_{*+2n}\widehat{L\dr Xk{*}}\cong\left\{\sum_{n\in\Z} a_nt^n:a_n\in\pi_{*+2n}\widehat{L\dr Xk*} \right\}\]
with the componentwise addition as stated in the introduction. We will now show how to describe the multiplication: Given $\left(\sum_{n\in\Z}a_nt^n\right),\left(\sum_{n\in\Z}b_nt^n\right)\in\HP_*(X/k)$, then we know that
\begin{equation}\label{cn}\left(\sum_{n\in\Z}a_nt^n\right)\cdot\left(\sum_{n\in\Z}b_nt^n\right)=\sum_{n\in\Z}c_nt^n\end{equation}
for some $c_n\in\pi_*\widehat{L\dr Xk*}$, so that we want to describe these coefficients $c_n$.

\begin{construction}
The graded ring $\pi_*\widehat{L\dr Xk*}$ can be equipped with the coarsest topology making all maps $\pi_*\widehat{L\dr Xk*}\to\pi_*L\dr Xk{\leq m}$ continuous for the discrete topology on the target. Concretely, this means a neighborhood basis of 0 is given by the kernels of these maps above. In particular, the topology cannot separate points that lie in every single such kernel, i.e. lie in the kernel of the surjective map $\pi_*\widehat{L\dr Xk*}\to\lim \pi_*L\dr Xk{\leq m}$. In degree $i$ this is precisely given by $\lim^1\pi_{i+1}L\dr Xk{\leq m}$. In fact $\lim\pi_*L\dr Xk{\leq m}$ is the "Hausdorffization" of this non-Hausdorff topology.
\end{construction}

\begin{lemma}
	In the equation \eqref{cn} the coefficient $c_n$ is a limit of the net $\sum_{i+j=n}a_i\cdot b_j$.
\end{lemma}
\begin{proof}
	It is enough to prove this statement for homogeneous elements, and for simplicity assume that $(\sum_{n\in\Z}a_nt^n)$ and $(\sum_{n\in\Z}b_nt^n)$ are both in degree 0. For the general case, one only has to correctly modify the degrees of elements, the arguments are the same.
	
	By definition of the topology on $\pi_*\widehat{L\dr Xk*}$ we have to show, that $c_n-\sum_{(i,j)\in{J_n}}a_i\cdot b_j$ for finite $J_n\subset\{i+j=n\}$ eventually lies in the kernel of the maps $\pi_*\widehat{L\dr Xk*}\to\pi_*L\dr Xk{\leq m}$. By Proposition \ref{com} these maps assemble to ring maps 
	\[\varphi_n\colon\HP_*(X/k)\to \pi_*L\dr Xk{\leq m}\llrrpar t,\]
	where we understand the multiplication of Laurent-series on the target. Moreover, because the coefficients of the target are in degrees $\geq -m$ as a graded ring, we even know, that $a_i\cdot b_j$ is sent to 0 in $\pi_*L\dr Rk{\leq m}$ as soon as $i<m/2$ or $j<m/2$. That means for every family of finite sets $J_n\subset\{i+j=n\}$ containing $I_n:=\{i+j=n:i,j\geq m/2\}$
	\begin{align*}\varphi_n\left(\sum_{n\in\Z}\left(\sum_{J_n}a_i\cdot b_j\right)t^n\right)&=\sum_{n\in\Z}\left(\sum_{I_n}\varphi_n(a_i)\cdot\varphi_n(b_j)\right)t^n\\
	&=\left(\sum_{n\in\Z}\varphi_n(a_n)t^n\right)\cdot\left(\sum_{n\in\Z}\varphi_n(b_n)t^n\right)\end{align*}
	But also by definition we have $\varphi_n\left(\sum_{n\in\Z}c_nt^n\right)=\left(\sum_{n\in\Z}\varphi_n(a_n)t^n\right)\cdot\left(\sum_{n\in\Z}\varphi_n(b_n)t^n\right)$. In particular, taking the difference and restricting again to single coefficients $c_n-\sum_{J_n}a_i\cdot b_j$ is sent to 0 in $\pi_*L\dr Xk{\leq m}$.
\end{proof}
This concludes the description of $\HP_*(X/k)$ given in the introduction.
\appendix

\section{HP of Schemes}\label{scheme}

In this section, we want to carefully describe the extension of Hochschild and periodic cyclic homology to (derived) schemes. 
We will refer to \cite{Lur18} [Section 1.1], \cite{Lur11} and \cite{To14} for an introduction to derived schemes over animated commutative (aka simplicially commutative) rings. We will only sketch the definition:

\begin{definition}
	For an animated commutative $k$-algebra $R$, define the affine derived scheme $\Spec R$ to be the pair $(|\Spec R|,\O_{\Spec R})$ where $|\Spec R|=|\Spec\pi_0R|$ is a topological space and $\O_{\Spec R}$ is a $\CAlg^\an_{k/}$-valued sheaf on $|\Spec R|$ with $\O_{\Spec R}(D(f))\simeq R[f^{-1}]$ for every elementary open $D(f)\subset |\Spec \pi_0R|$.\footnote{The existence of $\Spec R$ is deduced in \cite{Lur18} from Proposition \ref{extend_condition} below.}
	
	A general pair $X=(|X|,\O_X)$ with $|X|$ a topological space and $\O_X\in\Shv_{\CAlg^\an_{k/}}(|X|)$ is called a derived scheme, if there exist an open cover $\U$ of $X$, such that for all $U\in\U$ we have $(U,\O_X|_U)\cong \Spec R$\footnote{Under the appropriate notion of equivalence.} for some $R\in\CAlg^\an_{k/}$.
\end{definition}

\begin{remark}
	This notion generalizes ordinary schemes. In particular given a derived scheme $X$, the underlying ringed space $\pi_0X:=(|X|,\pi_0\O_X)$ is an ordinary scheme and we call a derived scheme $X$ affine\footnote{In fact $X$ is affine, if and only if $X=\Spec R$ for $R\in\CAlg^\an_{k/}$.}, quasi-affine, quasi-compact resp. quasi-separated if $\pi_0X$ is so.
\end{remark}

\begin{definition}\label{sheaf}
	Let $X$ be a derived $k$-scheme. A Zariski-sheaf with values in a category $\C$ on $X$ is a $\C$-valued sheaf on the topological space $|X|$, i.e. a functor $\F\colon\U(X)^\op\to \C$ from the opposite of the poset $\U(X)$ of opens of $|X|$, satisfying
	\[\F(U)\simeq\lim_{\begin{subarray} c\emptyset\neq S\subset I\\\text{finite}\end{subarray}}\F(U_S)\]
	for every $U=\bigcup_{i\in I}U_i\in\U(X)$ and with $U_S=U_{i_0}\cap\ldots\cap U_{i_k}$ for $S=\{i_0,\ldots,i_k\}$.
\end{definition}

Given a derived scheme $X$ over $k$ the goal is now to upgrade the functors $\HH(-/k),$ $\HP(-/k)\colon$ $\CAlg^\an_{k/}\to \Mod_k$ to Zariski-sheaves $\shvHH_k$ and $\shvHP_k$ on $X$ in order to define $\HH(X/k):=\Gamma(X,\shvHH_k)$ and $\HP(X):=\Gamma(X,\shvHP_k)$.

\begin{proposition}\label{extend_condition}
	Given a topological space $X$ and $\mathcal U_e$ a set of open subsets of $X$, such that
	\begin{enumerate}
		\item $\mathcal U_e$ forms a basis of the topology of $X$,
		\item $\mathcal U_e$ is closed under intersections.
	\end{enumerate}
	Then the adjunction $\begin{tikzcd}\Fun(\mathcal U(X)^\op,\C)\ar[r,shift left = 1, "\mathrm{res}"]&\Fun(\mathcal U_e^\op,\C)\ar[l, shift left=1, "\mathrm{Ran}"]\end{tikzcd}$ restrict to an equivalence of sheaf categories
	$\Shv_\C(X)\xrightarrow\sim\Shv_\C(\mathcal U_e)$
	with the induced Grothendieck topology on $\U_e$. If, moreover, $\mathcal U_e$ consist of quasi-compact opens,
	then $\Shv_\C(X)\simeq\Fun'(\U_e^\op,\C)$, where the right hand side consists of those presheaves $\F\colon\U_e^\op\to C$, that satisfy $\F(\emptyset)=0$ and $\F(U\cup V)\simeq \F(U)\times_{\F(U\cap V)}\F(V)$ for $U,V,U\cup V\in\U_e$.
\end{proposition}
\begin{proof}
	The first statement is a special case of the infinity categorical comparison Lemma for Grothendieck sites proven in \cite{Hoy14} Lemma C.3, and the second claim is \cite{Lur18} Proposition 1.1.4.4.
\end{proof}

We now do the standard procedure of extending an algebraic functor $\CAlg^\an_{k/}\to \C$ to a sheaf on geometric objects. We proceed in steps:

\begin{lemma}
	Given a quasi-affine derived scheme $X$ over $k$, there are $\Mod_k$-valued sheaves $\shvHH_k$ and $\shvHP_k$ on $X$, extending $\HH(-/k)$ and $\HP(-/k)$, i.e. for all affine open derived subschemes $U\subset X$, the sheaves recover Hochschild homology, resp. periodic cyclic homology:
	\[\Gamma(U,\shvHH_k)\simeq\HH(\O_X(U)/k)\qquad \Gamma(U,\shvHP_k)\simeq\HP(\O_X(U)/k)\]
\end{lemma}
\begin{proof}
	Set $\U_e$ to be the set of affine open derived subschemes of $X$. Then $\HH(-/k)$ and $\HP(-/k)$ give functors $\U_e^\op\to \Mod_k$ and let $\shvHH_k$ and $\shvHP_k$ denote their right Kan extension along $\U_e^\op\to\U(X)^\op$. We want to argue, that these are already Zariski-sheaves on $X$. 
	
	Because $X$ is quasi-affine, intersections of affines are computed in a surrounding affine derived scheme, and are affine again. The collection $\U_e$, thus, satisfies the conditions 1), 2) of Proposition \ref{extend_condition} and contains only quasi-compact opens, so that we are reduced to checking that the functors $\HH(-/k)$ and $\HP(-/k)$ satisfy the finite limit condition of $\Fun'(\U_e^\op,\Mod_k)$. As the Tate-construction commutes with finite limits, it is enough to only show the claim for Hochschild homology.
	
	For $R\in\CAlg_{k/}^\an$ the natural map $R\to\HH(R/k)$ in $\CAlg_k$ equips $\HH(R/k)$ with a module structure over $R$, such that for a map of animated commutative rings $R\to R'$ the functoriality induces a map $\HH(R/k)\otimes_RR'\to\HH(R'/k)$ in $\Mod_{R'}$. Now if $U\subset X$ is an affine open derived subscheme of $X$, then for every other affine open $V\subset U$ this map 
	\[\HH(\O_X(U)/k)\otimes_{\O_X(U)}\O_X(V)\to\HH(\O_X(V)/k)\]
	is an equivalence. Indeed, it suffices to check this locally on $V$, so we can reduce to distinguished opens $D(f)\subset V\subset U$ for $f\in\pi_0\O_X(U)$. But using that $\HH(-/k)$ commutes with filtered colimits we can identify both sides with $\HH(\O_X(U)[f^{-1}])$.
	
	Finally, assume that $F:I\to\U_e$ is a finite diagram with colimit $U$ as appearing in Proposition \ref{extend_condition}, then by the above $\HH(\O_X(F^\op(-))/k)\simeq\HH(\O_X(U)/k)\otimes_{\O_X(U)}\O_X(F^\op(-))$ and we win as tensoring is exact and $\O_X(F^\op(-))$ is a finite limit diagram due to the sheaf condition of $\O_X$ (using Proposition \ref{extend_condition} in the other direction).
\end{proof}

\begin{lemma}
	Given an arbitrary derived $k$-scheme $X$, we can furthermore extend Hochschild and periodic cyclic homology to sheaves $\shvHH_k$ and $\shvHP_k$ on $X$. Moreover, for all open qcqs derived subschemes $U\subset X$ we have
	\[\Gamma(U,\shvHP_k)\simeq \Gamma(U,\shvHH_k)^{tS^1}\]
\end{lemma}
\begin{proof}
	Let $\U_e$ now be the set of quasi-affine open derived subschemes of $X$, which satisfies 1) and 2) of Proposition \ref{extend_condition}. By the last Lemma $\HH(-/k)$ and $\HP(-/k)$ extend to sheaves on $\U_e^\op$ and by Proposition \ref{extend_condition} thus further extend to sheaves on entire $X$.
	
	Now take $U\subset X$ a quasi-compact quasi-separated derived open subscheme. Because of quasi-compactness there exist a finite open cover of $U$ by affine open subschemes $U_1,\ldots U_n$ and by the sheaf condition we get
	\[\Gamma(U,\shvHP_k)\simeq\lim_{S\subset[1,n]}\Gamma(U_S,\shvHP_k)\]
	in the notation of Definition \ref{sheaf}. Each $U_S$ is now quasi-affine as an open derived subscheme of an affine and quasi-compact by the quasi-separatedness of $U$. Thus, because the limit above is finite, it satisfies to check the claim for $U$ quasi-compact quasi-affine. Again, choosing a finite open cover by affines and using that the intersection of affines in quasi-affines is affine again, we can even reduce to the case that $U$ is an affine open. But in this case 
	\[\Gamma(U,\shvHP_k)\simeq\HP(\O_X(U)/k)\simeq\HH(\O_X(U)/k)^{tS^1}\simeq\Gamma(U,\shvHH_k)^{tS^1}\qedhere\]
\end{proof}

\begin{remark}\label{qcqs_finite}
	The proof of the last Lemma shows even more: For any sheaf $\F$ on a derived scheme $X$, the sections $\Gamma(U,\F)$ over a qcqs open derived subscheme $U$ are computed as a finite limit of the values of $\F$ on affines.
\end{remark}

\section{Tate Filtration}\label{Tatefiltration}
In this section we want to review the construction of the classical Tate-filtration introduced in \cite{GM95}. This content is not new and also recently has been explained in \cite{BL22} section 6.1. We would like to particular put a focus on multiplicative structures.

\begin{definition}
	Given a representation $\rho\colon S^1\to\GL(V)$ of $S^1$, the representation sphere $S^V$ is the one-point compactification of $V$. Furthermore we define $\S^V:=\Sigma^\infty S^V$ as the suspension spectrum of the representation sphere.
\end{definition}

\begin{remark}
	Note that if $V$ is finite dimensional there immediately is an equivalence $S^V\simeq S^{\dim_\R V}$, so that the homotopy type of $S^V$ only depends on the dimension of $V$. However, the $S^1$-action really uses the representation $S^1\to \GL(V)$.
\end{remark}

\begin{example}\label{standard representation}
	For $V=\Cx$ there is the standard representation given by $S^1\simeq U(1)\hookrightarrow \Cx^\times$. Its representation sphere sits in the pushout
	\[\begin{tikzcd}
	S^1\ar[r]\ar[d]&*\ar[d]\\
	*\ar[r]&S^V
	\end{tikzcd}\]
	with $S^1$-acting freely on itself. Thus, after adding basepoints to the top row $\Sigma^\infty$ gives a fibre sequence $\S[S^1]:=\Sigma^\infty_+S^1\to \S\to \S^V$ of spectra with $S^1$-action.
\end{example}

\begin{construction}
	Let $V$ be a finite dimensional representation of $S^1$. The map $0\to V$ of representations induces a sequence
	\[0\to V\to V\oplus V\to V\oplus V\oplus V\to\cdots\]
	which translates to the representation sphere spectra to a $\Z$-graded filtration
	\begin{equation}\label{repfiltration}
	\S^{\bullet V}:=\cdots\to \S^{-2V}\to\S^{ -V}\to\S\to \S^{V}\to\S^{ 2V}\to\S^{ 3V}\to \cdots
	\end{equation}
	where $\S^{ -nV}:=DS^{ nV}$ is the Spanier-Whitehead dual.
	Now if $V\neq 0$ all maps have to be non-equivariantly nullhomotopic, but this is definitely not the case with respect to the $S^1$-action. We will see this later in Proposition \ref{convergence to tate}.
\end{construction}
\begin{definition}
	Given a spectrum $X\in\Sp^{BS^1}$ with $S^1$-action, we define the Tate-filtration $\Fil_T X^{tS^1}$ as
		\[\cdots {\to} \left(\S^{  -2V}{\otimes} X\right)^{hS^1} {\to}\left(\S^{  -V}{\otimes} X\right)^{hS^1} {\to} X^{hS^1} {\to} \left(\S^{V}\otimes X\right)^{hS^1} {\to}\left(\S^{  2V}{\otimes} X\right)^{hS^1} {\to} \cdots\]
	for $V$ the standard representation of $S^1$ constructed in Example \ref{standard representation}.
\end{definition}
This definition would not be sensible if this would not give a filtration on $X^{tS^1}$ and we are bound to prove:

\begin{proposition}\label{convergence to tate}
	For $X\in\Sp^{BS^1}$ the Tate filtration $\Fil_T^\bullet X^{tS^1}$ is complete with underlying object $X^{tS^1}$.
\end{proposition}
\begin{proof}
	Because homotopy fixed points, as a limit, preserve completeness it satisfies to prove that $\lim \S^{-nV}\otimes X\simeq 0$ as a spectrum. But here we can compute
	\[\lim\S^{-nV}\otimes X\simeq \lim \map\left(\S^{nV},X\right)\simeq \map\left(\colim \S^{nV},X\right)\simeq 0
	\]
	because the colimit goes along nullhomotopic maps. However, as already indicated, those maps are not equivariantly nullhomotopic In fact every map
	\[\S^{-nV}\otimes X\to \S^{(-n+1)V}\otimes X\]
	induces an equivalence on the $S^1$-Tate construction. Indeed, via Example \ref{standard representation} we can identify the fibre as $\S^{-nV}\otimes \S[S^1]$, which is an induced $S^1$-spectrum, such that Tate vanishes on this fibre. Now we can look at the $\Z$-indexed fibre sequences defining the Tate constructions of $\S^{-nV}\otimes X$:
	\[\begin{tikzcd}
	\Sigma\left(\S^{-nV}\otimes X\right)_{hS^1}\ar[r]&\underbrace{\left(\S^{-nV}\otimes X\right)^{hS^1}}_{\simeq \Fil^n_TX^{tS^1}}\ar[r]&\left(\S^{-nV}\otimes X\right)^{tS^1}.
	\end{tikzcd}\]
	By the observation above the right hand filtration is constant at $X^{tS^1}$. The colimit of the left hand filtration vanishes, because commuting the colimit with $\Sigma(-\otimes X)_{hS^1}$ reduces again to computing a filtered colimit along nullhomotopic maps, which is 0. Thus together we see $\colim\Fil^n_T X^{tS^1}\simeq X^{tS^1}$. 
\end{proof}

We are interested in possible algebra structures on the Tate filtration. Because $(-)^{tS^1}$ is lax monoidal, $X^{tS^1}$ for an algebra $X\in\Alg(\Sp)$ inherits an algebra structure again. However, the question of algebra structures on $\Fil_T X^{tS^1}$ with respect to the Day convolution is more subtle.

We will use the following different description of the filtered category as a modules over a graded algebra. This insight comes from Lurie in \cite{Lur15} 3.2 and in this form is in \cite{Rak20} Proposition 3.2.9.

\begin{definition}
	For a stable symmetric monoidal category $\C$ with unit $\unit$, let $\unit[\beta]$ denote the underlying graded object of the unit in $\Fil(\C)$. It is a commutative algebra in $\Gr(\C)$ with underlying graded object $\bigoplus_{n\leq 0}\unit$. 
\end{definition}

Every object in the symmetric monoidal category $\Fil(\C)$ is canonically a module over the unit, such that the symmetric monoidal forgetful functor $\Fil(\C)\to\Gr(\C)$ refines to a functor $\Fil(\C)\to \Mod_{\unit[\beta]}(\Gr(\C))$. In fact remembering this action of $\unit[\beta]$ recovers the full filtered object:

\begin{theorem}[\cite{Rak20}]\label{cite1}
	For a symmetric monoidal stable category $\C$ the above functor $\Fil(\C)\to \Mod_{\unit[\beta]}(\Gr(\C))$ is an equivalence of symmetric monoidal categories.
\end{theorem}

In our situation we want to use this for $\C=\Sp_\Q^{BS^1}$ and show that the filtered object $\S^{-\bullet V}\otimes \Q$ is a commutative algebra in $\Fil(\Sp_\Q^{BS^1})$. 
There is also an algebraic description of this category due to Greenlees-Shipley \cite{GS09} in the non-Borel-complete setting and later as we use it here by \cite{MNN17}. Similar to above, because again in $\Sp_\Q$ every object carries a canonical module structure over the unit $\Q$, the lax functor $(-)^{hS^1}:\Sp_\Q^{BS^1}\to\Sp_\Q$ refines to a functor into $\Mod_{\Q^{hS^1}}(\Sp_\Q)$ and we have as a special case of Theorem 7.35 in \cite{MNN17}:

\begin{theorem}[\cite{MNN17}]\label{cite2}
	The functor $(-)^{hS^1}\colon\Sp_\Q^{BS^1}\to\Mod_{\Q^{hS^1}}(\Sp_\Q)$ is fully faithful with essential image given by those modules over $\Q^{hS^1}\simeq \Q\llrrbra t$ that are complete with respect to the $t$-adic filtration.
\end{theorem}

The Thom isomorphism over $\Q$ for complex vector bundles over $BS^1$ gives an $S^1$-equivariant equivalence of $\S^V\otimes \Q\simeq \Q[2]$ with trivial $S^1$-action on the right. Thus the map $\S\otimes \Q\to \S^V\otimes \Q\simeq \Q[2]$ in $\Sp_\Q^{BS^1}$ corresponds to an $\Q\llrrbra t$-module map $\Q\llrrbra t\to \Q\llrrbra t[2]$
for $|t|=-2$. In particular as a $\Q\llrrbra t$-module map it is determined by the image of 1 in $\Q\cdot t$. Because this map is not 0 as seen in the proof of Proposition \ref{convergence to tate}, up to a unit, it is given by multiplication by $t$. More generally this argument gives an identification of the image of the filtration $\S^{-\bullet V}\otimes \Q$ under $(-)^{hS^1}$ with the filtration

\[\cdots\xrightarrow{\cdot t} \Q\llrrbra t[-2n]\xrightarrow{\cdot t}\Q\llrrbra t\xrightarrow{\cdot t}\Q\llrrbra t[2n]\xrightarrow{\cdot t}\cdots\]
of $\Q\llrrbra t$-modules.

\begin{lemma}\label{mult}
	The filtration $\S^{-\bullet V}\otimes \Q$ can be given a commutative algebra structure in $\Fil(\Sp_\Q^{BS^1})$.
\end{lemma}
\begin{proof}
	Under the symmetric monoidal equivalences from the cited Theorems \ref{cite1} and \ref{cite2} we are reduced to equip the underlying graded object $\bigoplus_{n\in\Z}\Q\llrrbra t[-2n]$ of $(\S^{-\bullet V}\otimes \Q)^{hS^1}$ with a commutative algebra structure over $\Q\llrrbra t[\beta]$. To avoid confusion, let us introduce a formal variable $s$ in grading degree $-1$ and homological degree 2 to get an identification of underlying objects
	\[\bigoplus_{n\in\Z}\Q\llrrbra t[-2n]\simeq \Q\llrrbra t[s^{\pm 1}],\]
	which is the free graded commutative $\Q\llrrbra t$-algebra on the variables $s^{\pm 1}$. In particular sending $\beta$ to $s\cdot t$ gives $\Q\llrrbra t[s^{\pm 1}]$ the desired commutative algebra structure.
\end{proof}

\begin{proposition}
	For a commutative ring $k$ with $\Q\subset \pi_0 k$ and $R\in\CAlg_k^{BS^1}$ the filtration $\Fil_T R^{tS^1}$ permits the structure of a commutative algebra in $\Fil(\Mod_k)$.
\end{proposition}
\begin{proof}
	By construction $\Fil_T (-)^{tS^1}$ is the composite
	\[\Mod_k^{BS^1}\xrightarrow{(-)\otimes (\S^{-\bullet V}\otimes k)}\Fil(\Mod_k^{BS^1})\xrightarrow{(-)^{hS^1}}\Fil(\Mod_k).\]
	The second functor has a canonical lax structure. Because $k$ is a commutative algebra over $\Q$, also $(\S^{-\bullet V}\otimes k)$ inherits a commutative algebra structure via Lemma \ref{mult} and thus $\Fil_T(-)^{tS^1}$ refines to a lax symmetric monoidal functor and sends commutative algebras in $\Mod_k^{BS^1}$ to commutative algebras in $\Fil(\Mod_k)$.
\end{proof}

Given $C_*\in\CDGA_\Q\xhookrightarrow\iota\Sp_\Q^{BS^1}$ we would like to conclude this section with the comparison of the induced filtration on $\Fil_T^{\geq 0} UC_*^{tS^1}$ on the zeroth part $\Fil_T^0 UC_*^{tS^1}\simeq UC_*^{hS^1}$ to concrete filtrations on the chain level. 

\begin{proposition}\label{t-adic filtration}
	In the notation of Lemma \ref{computation}, there is an identification of the $t$-adic filtration on $UC_*^{hS^1}\simeq (UC_*\llrrbra t, td)$ with the Tate filtration $\Fil_T^{\geq 0} UC_*^{tS^1}$.
\end{proposition}
\begin{proof}
	Using the cocommutative bialgebra $A:= \Q[\epsilon]/\epsilon^2$ as defined in Construction \ref{A} we have the symmetric monoidal equivalence of categories $\Sp_\Q^{BS^1}\xrightarrow\sim\Mod_A\Sp_\Q$ (Corollary \ref{formalaction}). Therefore the filtration $\Fil^{\geq 0}_TUC_*^{tS^1}$ reads as
	\[\cdots \to\map_A(\Q,\Q[-4]\otimes C_*)\to\map_A(\Q,\Q[-2]\otimes C_*)\to\map_A(\Q,\Q\otimes C_*)\simeq UC_*^{hS^1}\]
	By duality this filtration is equivalently induced by the maps $\Q[2n]\to\Q[2n+1]$ from $\S^{-\bullet V}\otimes \Q$ for $n\geq 0$ in the first variable of the mapping spectrum. Choosing $P_*=(A\langle t^\vee\rangle, d_P)$ for $t^\vee$ primitive in degree 2 and $d_P(t^\vee)=\epsilon$ as in the proof of Lemma \ref{computation}, these maps
	\[
	\begin{tikzcd}
	P_*[2n]\ar[d]&\cdots\ar[r]&k\cdot (t^\vee)^2\ar[r,"\sim"]\ar[d]&k\cdot \epsilon t^\vee\ar[r,"0"]\ar[d]&k\cdot t^\vee\ar[r,"\sim"]\ar[d]&k\cdot \epsilon\ar[r,"0"]&k\\
		P_*[2n+1]&\cdots\ar[r]&k\cdot t^\vee\ar[r,"\sim"]&k\cdot \epsilon\ar[r,"0"]&k
	\end{tikzcd}
	\]
	are uniquely determined as $A$-module maps by the image of $t^\vee$. Because again the map is non-zero, the image of $t^\vee$ has to be a unit. In particular, up to isomorphism the maps
	\[C_*^{hS^1}[-2n-2]\to C_*^{hS^1}[-2n]\]
	are given by multiplication with the dual $t$ in $(C_*\llrrbra t, td)$.
\end{proof}

\printbibliography

\end{document}